\documentclass [12pt]{amsart}
\usepackage{amsmath,amssymb}
 \usepackage{amsthm, amsfonts}
 \usepackage{enumerate}
 \usepackage{amssymb,amsmath,amsthm, amsfonts}
 \usepackage{cite}
 \usepackage{xcolor}
\usepackage[pagewise]{lineno}%\linenumbers
\newtheorem{theorem}{Theorem}[section]

\newtheorem{corollary}[theorem]{Corollary}
\theoremstyle{definition}
\newtheorem{definition}[theorem]{Definition}
\newtheorem{example}[theorem]{Example}

\theoremstyle{remark}

\numberwithin{equation}{section}

\begin{document}

\title [{{On graded weakly $J_{gr}$-semiprime}}]{ On graded weakly $J_{gr}$-semiprime submodules }

 \author[{{M. Alnimer, K. Al-Zoubi and M. Al-Dolat}}]{\textit{Malak Alnimer, Khaldoun Al-Zoubi}* and Mohammed Al-Dolat  }

\address
{\textit{Malak Alnimer, Department of Mathematics and
Statistics, Jordan University of Science and Technology, P.O.Box
3030, Irbid 22110, Jordan.}}
\bigskip
{\email{\textit{mfalnimer21@sci.just.edu.jo}}}

\address
{\textit{ Khaldoun Al-Zoubi , Department of Mathematics and
Statistics, Jordan University of Science and Technology, P.O.Box
3030, Irbid 22110, Jordan.}}
\bigskip
{\email{\textit{kfzoubi@just.edu.jo}}}

\address
{\textit{Mohammed Al-Dolat , Department of Mathematics and
Statistics, Jordan University of Science and Technology, P.O.Box
3030, Irbid 22110, Jordan.}}
\bigskip
{\email{\textit{mmaldolat@just.edu.jo}}}

 \subjclass[2010]{13A02, 16W50.}

\date{}
\begin{abstract}
 Let $\Gamma$ be a group, $\Re$ be a $\Gamma$-graded commutative ring with unity $1$ and $\Im$ a graded $\Re$-module. In this
paper, we introduce the concept of graded weakly $J_{gr}$-semiprime
submodules as a generalization of graded weakly semiprime
submodules. We study several results concerning of graded weakly $J_{gr}$%
-semiprime submodules. For example, we give a characterization of graded
weakly $J_{gr}$-semiprime submodules. Also, we find some relations between
graded weakly $J_{gr}$-semiprime submodules and graded weakly semiprime
submodules. In addition, the necessary and sufficient condition for graded
submodules to be graded weakly $J_{gr}$-semiprime submodules are
investigated. A proper graded submodule $U$ of $\Im$ is said to be a graded weakly
$J_{gr}$-semiprime submodule of $\Im$ if whenever $r_{g}\in
h(\Re),$ $m_{h}\in h(\Im)$ and $n\in
%TCIMACRO{\U{2124} }%
%BeginExpansion
\mathbb{Z}
%EndExpansion
^{+}$ with $0\neq r_{g}^{n}m_{h}\in U$, then $r_{g}m_{h}\in U+J_{gr}(\Im)$, where $J_{gr}(\Im)$ is the graded Jacobson
radical of $\Im.$

 \end{abstract}
\keywords{graded weakly $J_{gr}$-semiprime submodule, graded $J_{gr}$-semiprime submodule, graded weakly semiprime submodule  \\
$*$ Corresponding author}
 \maketitle

%---------------------------------------------------------------------

%                Bec1:   Introduction
%-----------------------------------------------------------------------
\section{Introduction and Preliminaries}
    Throughout this work, we assume that $\Re $ is a commutative $\Gamma$-graded
    ring with identity and $\Im $ is a unitary graded $\Re $-module.

Let $\Gamma$ be a group. A ring $\Re $ is said to be a $\Gamma$-graded ring if there exist additive subgroups $\Re _{g}$ of $\Re $
indexed by the elements $g\in \Gamma$ with $\Re =\bigoplus_{g\in \Gamma}\Re _{g}$
and $\Re _{g}\Re _{h}\subseteq \Re _{gh}$ for all $g$, $h\in \Gamma$. We
set $h(\Re ):=\cup _{g\in \Gamma}\Re _{g}$. If $t\in \Re $, then $t$ can be
written uniquely as $\sum_{g\in \Gamma}t_{g}$, where $t_{g}$ is called a
homogeneous component of $t$ in $\Re _{g}$. Let $\Re =\bigoplus_{g\in \Gamma}\Re
_{g}$ be a $\Gamma$-graded ring. An ideal $L$ of $\Re $ is said to be a graded
ideal if $L=\bigoplus_{g\in \Gamma}(L\cap \Re _{g}):=\bigoplus_{g\in \Gamma}L_{g}$.
By $L\leq^{id} _{\Gamma}\Re $, we mean that $L$ is a graded ideal of $\Re $.
Also, by $L<^{id}_{\Gamma}\Re $, we mean that $L$ is a proper graded ideal of $
\Re$.
Let $\Re $ be a $\Gamma$-graded ring and $\Im $ an $\Re $-module. Then $%
\Im $ is a $\Gamma$-graded $\Re $-module if there exists a family of additive subgroups $\{\Im
_{g}\}_{g\in \Gamma}$ of $\Im $ with $\Im =\bigoplus_{g\in \Gamma}\Im _{g}$ and $%
\Re _{g}\Im _{h}\subseteq \Im _{gh}$ for all $g,h\in \Gamma$. We set $%
h(\Im ):=\cup _{g\in \Gamma\ }\Im _{g}.$ Let $\Im =\bigoplus_{g\in \Gamma}\Im _{g}$ be
a graded $\Re$-module. A submodule $U$ of $\Im $ is said to be \textit{a
graded submodule of }$M$ \ if $U=\bigoplus_{g\in \Gamma}(U\cap \Im
_{g}):=\bigoplus_{g\in \Gamma}U_{g}$. By $U\leq^{sub} _{\Gamma}\Im $, we mean that $U$ is a $\Gamma$-graded submodule of $\Im $.
Also, by $U<^{sub}_{\Gamma}\Im $, we mean that $U$ is a proper $\Gamma$-graded submodule of $%
\Im .$ These basic properties and more information on graded rings and modules can be found in \cite{H, NV1,NV2, NV3}. A $<^{sub}_{\Gamma}\Im $ is said to be \textit{a }$Gr$\textit{-maximal }if there is a $L\leq^{sub} _{\Gamma}\Im $ with $U\subseteq L\subseteq
\Im $, then $U=L$ or $L=\Im ,$ (see \cite{NV3}). The graded Jacobson
radical of a graded module $\Im $, denoted by $J_{gr}(\Im )$, is
defined to be the intersection of all $Gr$-maximal submodules of $\Im $ (if $%
\Im $ has no $Gr$-maximal submodule then we shall take, by definition, $%
J_{gr}(\Im )=\Im )$, (see \cite{NV3}).

\bigskip

The concept of graded semiprime submodules has been studied and introduced
by many authors (e.g. \cite{KAb,FGW,FG,LV}). A $U<^{sub}_{\Gamma}\Im $
is called \textit{a graded semiprime }(briefly, $Gr$-semiprime) \textit{%
submodule }if whenever $t_{g}\in h(\Re )$, $m_{h}\in h(\Im )$ and $n\in
\mathbb{Z}^{+}$ with $\ t_{g}^{n}m_{h}\in U$, then $t_{g}m_{h}\in U$, (see
\cite{FG}). In \cite{TZ}, the authors present graded weakly semiprime submodules, which
are generalized from graded semiprime submodules. A $U<^{sub}_{\Gamma}\Im $ is called \textit{a graded weakly semiprime
(briefly, }$Gr$\textit{-}$W$\textit{-semiprime) submodule }if whenever $%
t_{g}\in h(\Re )$, $m_{h}\in h(\Im )$ and $n\in \mathbb{Z}^{+}$ with $0\neq
\ t_{g}^{n}m_{h}\in U$, then $t_{g}m_{h}\in U$. Our paper introduces graded
weakly $J_{gr}$-semiprime submodules as a generalization of $Gr$-$W$%
-semiprime submodules. Several results concerning graded weakly $J_{gr}$%
-semiprime submodules will be discussed.
%--------------------------------------------------------------------
%-----------------------section 22222222222222222--------resultes
 \section{Results}
 %--------------------------------------------------
 %-------------------Definition 2.1--------------

\begin{definition}
 A proper graded submodule $U$ of $\Im$ is said to be a graded weakly $J_{gr}$-semiprime (briefly, $Gr$-$W$-$%
J_{gr} $-semiprime) submodule of $\Im$ \ if whenever $0\neq
r_{g}^{n}m_{h}\in U $ where $r_{g}\in h(\Re)$, $m_{h}\in h(\Im)$ and $n\in
%TCIMACRO{\U{2124} }%
%BeginExpansion
\mathbb{Z}
%EndExpansion
^{+}$, then $r_{g}m_{h}\in U+J_{gr}(\Im)$. In particular, a graded ideal $L$\ of $\Re $\ is said to be a
graded weakly $J_{gr}$-semiprime ideal of $\Re $ if  $L$ is a graded weakly
$J_{gr}$-semiprime submodule of the graded $\Re $-module $\Re $.
\end{definition}

%---------------------------------example 2-2---------------------------
It is clear that every $Gr$-$W$-semiprime submodule is a $Gr$-$W$-$J_{gr}$-semiprime submodule of $\Im$, but the converse is not true in general. This
is clear from the following example.
\begin{example}
Let $\Gamma=%
%TCIMACRO{\U{2124} }%
%BeginExpansion
\mathbb{Z}
%EndExpansion
_{2}$ and $\Re =%
%TCIMACRO{\U{2124} }%
%BeginExpansion
\mathbb{Z}
%EndExpansion
$ be a $\Gamma$-graded ring with $\Re _{0}=%
%TCIMACRO{\U{2124} }%
%BeginExpansion
\mathbb{Z}
%EndExpansion
$, $\Re _{1}=\{0\}$. Then $\Im =%
%TCIMACRO{\U{2124} }%
%BeginExpansion
\mathbb{Z}
%EndExpansion
_{24}$ is a graded $\Re $-module with $\Im _{0}=%
%TCIMACRO{\U{2124} }%
%BeginExpansion
\mathbb{Z}
%EndExpansion
_{24}$ and $\Im _{1}=\{\overline{0}\}$. Let $U=\{\overline{0},\overline{8},%
\overline{16}\}$  $\leq^{sub} _{\Gamma} $ $%
%TCIMACRO{\U{2124} }%
%BeginExpansion
\mathbb{Z}
%EndExpansion
_{24}$. Since $J_{gr}(%
%TCIMACRO{\U{2124} }%
%BeginExpansion
\mathbb{Z}
%EndExpansion
_{24})=\langle \overline{2}\rangle \cap \langle \overline{3}\rangle =\langle
\overline{6}\rangle =\{\overline{0},\overline{6},\overline{12},\overline{18}%
\}$ and whenever $0\neq r^{k}m\in U$\ for $r\in h(%
%TCIMACRO{\U{2124} }%
%BeginExpansion
\mathbb{Z}
%EndExpansion
)$, $m\in h(%
%TCIMACRO{\U{2124} }%
%BeginExpansion
\mathbb{Z}
%EndExpansion
_{24})$, $k\in
%TCIMACRO{\U{2124} }%
%BeginExpansion
\mathbb{Z}
%EndExpansion
^{+}$ implies that $rm\in U+J_{gr}(%
%TCIMACRO{\U{2124} }%
%BeginExpansion
\mathbb{Z}
%EndExpansion
_{24})=\{\overline{0},\overline{8},\overline{16}\}+\{\overline{0},\overline{6%
},\overline{12},\overline{18}\}=\langle \overline{2}\rangle ,$ we have $U$
is a $Gr$-$W$-$J_{gr}$-semiprime submodule of $\Im $. How ever, $U$ is not $%
Gr$-$W$-semiprime submodule of $\Im $ since there exist $2\in h(%
%TCIMACRO{\U{2124} }%
%BeginExpansion
\mathbb{Z}
%EndExpansion
)$, $\overline{2}\in h(%
%TCIMACRO{\U{2124} }%
%BeginExpansion
\mathbb{Z}
%EndExpansion
_{24})$ and $2\in
%TCIMACRO{\U{2124} }%
%BeginExpansion
\mathbb{Z}
%EndExpansion
^{+}$ such that $0\neq 2^{2}\cdot \overline{2}=\overline{8}\in U$ \ but $%
2\cdot \overline{2}=\overline{4}\notin U$.
\end{example}

 %--------------------------------------------------
 %-------------------Theorem 2.3--------------
Following are theorems that give some equivalent characterizations
of $Gr$-$W$-$J_{gr}$-semiprime submodule.

\begin{theorem}\label{th2.3}
Let $U<^{sub}_{\Gamma}\Im$. Then the following statements are equivalent.
  \begin{enumerate}[\upshape (i)]
\item  $U$ is a $Gr$-$W$-$J_{gr}$-semiprime submodule of $\Im$.
\item  $(U:_{\Im}\langle r_{g}^{n}\rangle )\subseteq (\langle 0\rangle
:_{\Im}\langle r_{g}^{n}\rangle )\cup (U+J_{gr}(\Im):_{\Im}\langle
r_{g}\rangle )$, for each $r_{g}\in h(\Re)$.
\item  Either $(U:_{\Im}\langle r_{g}^{n}\rangle )\subseteq (\langle 0\rangle
:_{\Im}\langle r_{g}^{n}\rangle )$ or $(U:_{\Im}\langle r_{g}^{n}\rangle
)\subseteq (U+J_{gr}(\Im):_{\Im}\langle r_{g}\rangle )$, for each $r_{g}\in
h(\Re)$.
\end{enumerate}
\end{theorem}

\begin{proof}
 $(i)\rightarrow (ii)$: Let $r_{g}\in h(\Re )$ and $m_{h}\in (U:_{\Im }\langle r_{g}^{n}\rangle
)\cap h(\Im ).$ Then $\langle r_{g}^{n}\rangle m_{h}\subseteq U$ and hence $%
r_{g}^{n}m_{h}\in U$. If $r_{g}^{n}m_{h}\neq 0,$ then$\ r_{g}m_{h}\in
U+J_{gr}(\Im )$ as $U$ is $Gr$-$W$-$J_{gr}$-semiprime submodule of $\Im $.
Hence $\langle r_{g}\rangle m_{h}\subseteq U+J_{gr}(\Im )$, it follows that $%
m_{h}\in (U+J_{gr}(\Im ):_{\Im }\langle r_{g}\rangle )$. Thus $m_{h}\in
(\langle 0\rangle :_{\Im }\langle r_{g}^{n}\rangle )\cup (U+J_{gr}(\Im
):_{\Im }\langle r_{g}\rangle )$. If $r_{g}^{n}m_{h}=0,$ then $\langle
r_{g}^{n}\rangle m_{h}\subseteq \{0\}$ and so $m_{h}\in (\langle 0\rangle
:_{\Im }\langle r_{g}^{n}\rangle )$. Hence $m_{h}\in (\langle 0\rangle :_{\Im
}\langle r_{g}^{n}\rangle )\cup (U+J_{gr}(\Im ):_{\Im }\langle r_{g}\rangle )
$. Therefore $(U:_{\Im }\langle r_{g}^{n}\rangle )\subseteq (\langle
0\rangle :_{\Im }\langle r_{g}^{n}\rangle )\cup (U+J_{gr}(\Im ):_{\Im
}\langle r_{g}\rangle )$.

$(ii)\rightarrow (iii)$: It is clear.
% Let $r_{g}\in h(\Re )$. Suppose that $(U:_{\Im
%}\langle r_{g}^{n}\rangle )\nsubseteq (\langle 0\rangle :_{\Im }\langle
%r_{g}^{n}\rangle )$, we need to prove $(N:_{\Im }\langle r_{g}^{n}\rangle
%)\subseteq (N+J_{gr}(\Im ):_{\Im }\langle r_{g}\rangle )$. Let $x_{h}\in (U:_{\Im }\langle r_{g}^{n}\rangle )\cap h(\Im ).$\
%Then by hypothesis (2), we get $x_{h}\in (\langle 0\rangle :_{\Im }\langle
%r_{g}^{n}\rangle )\cup (U+J_{gr}(\Im ):_{\Im }\langle r_{g}\rangle )$, that
%is either $x_{h}\in (\langle 0\rangle :_{\Im }\langle r_{g}^{n}\rangle )$\
%or $x_{h}\in (U+J_{gr}(\Im ):_{\Im }\langle r_{g}\rangle )$. \ If $x_{h}\in
%(\langle 0\rangle :_{\Im }\langle r_{g}^{n}\rangle )$, then $(U:_{\Im
%}\langle r_{g}^{n}\rangle )\subseteq (\langle 0\rangle :_{\Im }\langle
%r_{g}^{n}\rangle )$ as a contradiction with assumption so $x_{h}\in
%(U+J_{gr}(\Im ):_{\Im }\langle r_{g}\rangle )$. Therefore $(U:_{\Im }\langle
%r_{g}^{n}\rangle )\subseteq (U+J_{gr}(\Im ):_{\Im }\langle r_{g}\rangle )$.

$(iii)\rightarrow (i)$: Let $r_{g}\in h(\Re )$, $m_{h}\in h(\Im )$ and $%
n\in
%TCIMACRO{\U{2124} }%
%BeginExpansion
\mathbb{Z}
%EndExpansion
^{+}$ with $0\neq r_{g}^{n}m_{h}\in U$. Then $\{0\}\neq \langle
r_{g}^{n}\rangle m_{h}\subseteq U$, which implies that $m_{h}\in (U:\langle
r_{g}^{n}\rangle )$ and $m_{h}\notin (\langle 0\rangle :_{\Im }\langle
r_{g}^{n}\rangle ).$ Now by (iii), we get $m_{h}\in (U+J_{gr}(\Im ):_{\Im
}\langle r_{g}\rangle )$ and so $r_{g}m_{h}\in U+J_{gr}(\Im )$. Therefore, $U
$ is a $Gr$-$W$-$J_{gr}$-semiprime submodule of $\Im $.
\end{proof}

%--------------------------------------------------------------------------------------------
%--------------------------------theorem  2.4 ----------------------------------------------
\begin{theorem}\label{th2.4}
 Let $%
U<^{sub}_{\Gamma}\Im $. Then the following statements are equivalent.
 \begin{enumerate}[\upshape (i)]
\item $U$ is a $Gr$-$W$-$J_{gr}$-semiprime submodule of $\Im$.

\item  For every $K\leq^{sub} _{\Gamma}\Im,$ $r_{g}\in h(\Re)$ and $n\in
%TCIMACRO{\U{2124} }%
%BeginExpansion
\mathbb{Z}
%EndExpansion
^{+}$ with $\{0\}\neq \langle r_{g}\rangle ^{n}K\subseteq U,$ then $\langle
r_{g}\rangle K\subseteq U+J_{gr}(\Im)$.
 \end{enumerate}
\end{theorem}
\begin{proof}

$(i)\Rightarrow (ii)$ Let $K$ $\leq^{sub} _{\Gamma}\Im,$ $r_{g}\in h(\Re)$ and $n\in
%TCIMACRO{\U{2124} }%
%BeginExpansion
\mathbb{Z}
%EndExpansion
^{+}$ with $\{0\}\neq \langle r_{g}\rangle ^{n}K\subseteq U.$ This implies
that, $K\subseteq (U:_{\Im}\langle r_{g}^{n}\rangle )$ and $K\nsubseteq
(\langle 0\rangle :_{\Im}\langle r_{g}^{n}\rangle ).$ Since $U$ is a $Gr$-$W$%
-$J_{gr}$-semiprime submodule of $\Im$, by Theorem \ref{th2.3}, we have $%
K\subseteq $ $(U:_{\Im}\langle r_{g}^{n}\rangle )\subseteq
(U+J_{gr}(\Im):_{\Im}\langle r_{g}\rangle ),$ hence $\langle r_{g}\rangle
K\subseteq U+J_{gr}(\Im)$.

$(ii)\Rightarrow (i)$ Let $0\neq r_{g}^{n}m_{h}\in U$ where $r_{g}\in
h(\Re )$, $m_{h}\in h(\Im )$ and $n\in
%TCIMACRO{\U{2124} }%
%BeginExpansion
\mathbb{Z}
%EndExpansion
^{+}.$ Then $\{0\}\neq \langle r_{g}\rangle ^{n}\langle m_{h}\rangle \subseteq U$%
. Now by (ii), we have $\langle r_{g}\rangle \langle m_{h}\rangle \subseteq
U+J_{gr}(\Im ),$ it follows that $r_{g}m_{h}\in U+J_{gr}(\Im ).$
Therefore, $U$ is a $Gr$-$W$-$J_{gr}$-semiprime submodule of $\Im $
\end{proof}
%----------------------------------Cor cor 2-5
The following corollaries follow directly from Theorem \ref{th2.4}.
\begin{corollary}\label{co2.5}
Let $U<^{sub}_{\Gamma}\Im.$ Then $U$ is $Gr$%
-$W$-$J_{gr}$-semiprime submodule of $\Im$ if and only if for every $%
r_{g}\in h(\Re)$, and $n\in
%TCIMACRO{\U{2124} }%
%BeginExpansion
\mathbb{Z}
%EndExpansion
^{+}$ with $\{0\}\neq \langle r_{g}^{n}\rangle \Im\subseteq U$, then $%
\langle r_{g}\rangle \Im\subseteq U+J_{gr}(\Im)$.
\end{corollary}
%----------------------------------Cor cor 2-6
\begin{corollary}\label{co2.6}
Let $U<^{sub}_{\Gamma}\Im.$ Then $U$ is $%
Gr $-$W$-$J_{gr}$-semiprime submodule of $\Im$ if and only if for every $%
r_{g}\in h(\Re)$, $K\leq^{sub} _{\Gamma}\Im$ and $n\in
%TCIMACRO{\U{2124} }%
%BeginExpansion
\mathbb{Z}
%EndExpansion
^{+}$ with $\{0\}\neq r_{g}^{n}K\subseteq U$, then $r_{g}K\subseteq
U+J_{gr}(\Im)$.
\end{corollary}

%----------------------------------theorem  2-7
\begin{theorem}\label{th2.7}
 Let $U$ be a $Gr$-$W$-$J_{gr}$%
-semiprime submodule of $\Im$ with $J_{gr}(\Im)\subseteq U.$ Then $U$ is a $%
Gr$-$W$-semiprime submodule of $\Im$.
\end{theorem}
 \begin{proof}
 Let $r_{g}\in h(\Re )$, $m_{h}\in h(\Im )$ and $n\in
%TCIMACRO{\U{2124} }%
%BeginExpansion
\mathbb{Z}
%EndExpansion
^{+}$ with $0\neq r_{g}^{n}m_{h}\in U.$ Since $U$ is a $Gr$-$W$-$J_{gr}$%
-semiprime submodule of $\Im $ and $J_{gr}(\Im)\subseteq U$, we have $r_{g}m_{h}\in U+J_{gr}(\Im )=U$.
Therefore, $U$ is a $Gr$-$W$-semiprime submodule of $\Im $.
 \end{proof}
%----------------------------------theorem  2-8

Let $\Re $ be a $\Gamma $-graded ring and $\Im $, $\Im ^{^{\prime }}$ be two
graded $\Re $-modules. Let $\varphi :\Im \rightarrow \Im ^{^{\prime }}$ be
an $\Re $-module homomorphsim. Then $\varphi $ is said to be a graded
homomorphsim if $\varphi (\Im _{g})\subseteq \Im _{g}^{^{\prime }}$ for all $%
g\in \Gamma $, see \cite{NV3}.

\begin{theorem}\label{th2.8}
Let $U<_{\Gamma }^{sub}\Im $. If $J_{gr}(\Im /U)=\{U\}$, then $J_{gr}(\Im
)\subseteq U$.
\end{theorem}
\begin{proof}
 Define $\varphi :\Im \rightarrow \Im /U$ be a graded homeomorphism given by $%
\varphi (x)=x+U$ for all $x\in h(\Im )$, by \cite[ Theorem 2.12 (i)]{KAl}, $%
\varphi (J_{gr}(\Im ))\subseteq J_{gr}(\Im /U)$. Since $J_{gr}(\Im /U)=\{U\}$%
, then $\{U\}\subseteq \varphi (J_{gr}(\Im ))\subseteq \{U\}$, so we have $%
\varphi (J_{gr}(\Im ))=\{U\}$, thus $J_{gr}(\Im )\subseteq Ker\varphi =U$.
\end{proof}
%----------------------------------cor  2-9
\begin{corollary}\label{co2.9}
Let $U$ be a $Gr$-$W$-$J_{gr}$-semiprime submodule of $\Im$ with $J_{gr}(\frac{\Im}{U})=\{U\}$, then $U$
is a $Gr$-$W$-semiprime submodule of $\Im$.
\end{corollary}
\begin{proof}
This is clear by Theorem \ref{th2.8} and Theorem \ref{th2.7}.
\end{proof}

%-------------------------------------them 2.10-----------------
A graded $\Re $-module $\Im $ is a graded semisimple ($Gr$-semisimple) if
and only if every graded submodule $U$ of $\Im $ is a direct summand. That
is $\Im $ is a $Gr$-semisimple$\ $\ if and only if  for every graded
submodule $U$ of $\Im $ there exists $L$ a graded submodule of $\Im $ such
that $\Im =U\oplus L$. 

A graded submodule $U$ is called a graded small ($Gr$-small) if $\Im=U+V$ for
  $V$ $\leq^{sub} _{\Gamma}$ $\Im$ implies that $V=\Im$, see \cite{KQ}. 
\begin{theorem} \label{th2.10}
 Let $\Im $ be a $Gr$-semisimple $\Re $-module and $U$ be a $Gr$-$W$-$J_{gr}$-semiprime\
submodule of $\Im $. Then $U$ is a $Gr$-$W$-semiprime submodule of $\Im $.
\end{theorem}
\begin{proof}
 Let $\Im $ be a $Gr$-semisimple $\Re $-module, then every graded
submodule of $\Im $ is a direct summand. Thus the only $Gr$-small sumodule
of $\Im $ is $\{0\}$, it follows that $J_{gr}(\Im )=\sum \{S:S$ is a $Gr$-small submodule of $\Im \}=\{0\}\subseteq U$ by \cite[Theorem 2.10]{KAl}. Since $U$ is a $Gr$-$W$-$%
J_{gr}$-semiprime\ submodule of $\Im $, by Theorem \ref{th2.7}, then $U$ is a $Gr$-$W
$-semiprime submodule of $\Im $.
\end{proof}
%-----------------------------theorem 2.11------------------------------------------------
Recall from \cite{A} that a graded module $\Im$
is said to be graded torsion ($Gr$-torsion) free $\Re$-module if whenever $r_{g}m_{h}=0$ where $r_{g}\in h(\Re )$ and $%
m_{h}\in $h($\Im$), implies that either $m_{h}=0$ or $r_{g}=0$.

\begin{theorem} \label{th2.12}
Let $\Im$ be a $Gr$-torsion free $\Re$%
-module and $U\leq^{sub} _{\Gamma}\Im$ with $J_{gr}(\frac{\Im}{U})=\{U\}$. Then $U$ is
a $Gr $-$W$-$J_{gr}$-semiprime submodule of $\Im$ if and only if for any
nonzero $L$ $\leq^{id} _{\Gamma}$ $\Re$, $(U:_{\Im}L)$ is a $Gr$-$W$-$J_{gr}$%
-semiprime submodule of $\Im$.
\end{theorem}
\begin{proof}
$(\Longrightarrow )$ Let $0\neq L$ $\leq^{id} _{\Gamma}$ $%
\Re $, $m_{h}\in h(\Im )$, $r_{g}\in h(\Re )$  and $n\in
%TCIMACRO{\U{2124} }%
%BeginExpansion
\mathbb{Z}
%EndExpansion
^{+}$ with $0\neq r_{g}^{n}m_{h}\in (U:_{\Im }L).$ Then $\{0\}\neq \langle
r_{g}^{n}\rangle m_{h}\subseteq (U:_{\Im }L),$ and hence $\langle
r_{g}^{n}\rangle (Lm_{h})\subseteq U$. If $\langle r_{g}^{n}\rangle
(Lm_{h})=\{0\},$ so there exists $0\neq i\in L\cap h(\Re )$ with $%
\langle r_{g}^{n}\rangle im_{h}=\{0\},$ so $i\cdot r_{g}^{n}m_{h}=0$. Hence $%
r_{g}^{n}m_{h}=0$ as $\Im $ is a $Gr$-torsion free $\Re $-module, which is a
contradiction. So assume that $\langle r_{g}\rangle ^{n}(Lm_{h})=\langle
r_{g}^{n}\rangle (Lm_{h})\neq \{0\}.$ Since $U$ is a $Gr$-$W$-$J_{gr}$%
-semiprime submodule of $\Im ,$  by Theorem \ref{th2.4},  $\langle
r_{g}\rangle (Lm_{h})\subseteq U+J_{gr}(\Im ).$ But $J_{gr}(\frac{\Im }{U}%
)=\{U\}$, by Theorem \ref{th2.8}, we have $J_{gr}(\Im )\subseteq U$ so $%
\langle r_{g}\rangle (Lm_{h})\subseteq U.$ This implies that, $\langle
r_{g}\rangle m_{h}\subseteq (U:_{\Im }L)\subseteq (U:_{\Im }L)+J_{gr}(\Im )$
and hence $r_{g}m_{h}\in (U:_{\Im }L)\subseteq (U:_{\Im }L)+J_{gr}(\Im
).$ Therefore, $(U:_{\Im }L)$ is a $Gr$-$W$-$J_{gr}$-semiprime submodule of $%
\Im $.

$(\Longleftarrow )$ Assume that $(U:_{\Im }L)$ is a $Gr$-$W$-$J_{gr}$%
-semiprime submodule of $\Im $ for any nonzero $L$ $\leq^{id} _{\Gamma}$ $\Re $.
Put $\Re =L,$ then $U=$ $(U:_{\Im }\Re )$ is a $Gr$-$W$-$J_{gr}$-semiprime
submodule of $\Im .$
\end{proof}

%--------------------------------------2-12----------------------

Recall from \cite{ET} that a graded $\Re $-module $%
\Im $ is called a graded multiplication module ($Gr$-multiplication module)
if for every  $U$ $\leq^{sub} _{\Gamma}$ $\Im $ there exists a  $K$ $\leq^{id} _{\Gamma}$
$\Re $ such that $U=K\Im .$ If $\Im $ is $Gr$-multiplication $\Re $-module, $U=(U:_{\Re }\Im )\Im $, for
every $U$ $\leq^{sub} _{\Gamma}$ $\Im$.
 
 The set of all homogeneous zero divisors of $\Re $ is $G$-$Z(\Re )=\{r\in
h(\Re ):$ $rs=0$ for some $0\neq s\in h(\Re )\}$, and the set of all
homogeneous regular element is $G$-$C(\Re )=\{c\in
h(\Re ):c\notin G$-$Z(\Re )\}=\{c\in h(\Re ):cr\neq 0$ for all $0\neq r\in
h(\Re )\}$. It is clear that $\Im $ is a $Gr$-torsion free if and only if  $%
cm\neq 0$ for all $c\in G$-$C(\Re )$ and $0\neq m\in h(\Im )$.

\begin{theorem} \label{th2.13}
Every faithful $Gr$-multiplication $\Re $-module is a $Gr$-torsion free.
\end{theorem}
\begin{proof}
Suppose not, $\Im $ is not $Gr$-torsion free. Hence there exist $c\in G$-$C(\Re )$ and $0\neq m\in h(\Im )$ with $cm=0.$ Since $\Im $ is a faithful $Gr$-multiplication $\Re $-module, there exists a  $
L$ $\leq^{id} _{\Gamma}$ $\Re$ with $\Re m=L\Im ,$ and so $\Re cm=cL\Im .$ This implies that $(cL)\Im
=\{0\}$. Since $\Im $ is a faithful, $cL=\{0\}$. Hence $c\in G$-$Z(\Re )$ since $L\neq0$ and so $c\notin G$-$C(\Re )$ which is a contradiction. Therefore, $\Im $ is a $Gr$
-torsion free.\end{proof}

%-------------------------------------------2-13-----------------------------------
\begin{corollary} \label{co2.14}
Let $\Im$ be a faithful $Gr$%
-multiplication $\Re$-module and $U\leq^{sub} _{\Gamma}\Im$, with $J_{gr}(\frac{\Im}{U}%
)=\{U\} $. Then $U$ is a $Gr$-$W$-$J_{gr}$-semiprime submodule of $\Im$ if
and only if for any nonzero $L$ $\leq^{id} _{\Gamma}$ $\Re$, $(U:_{\Im}L)$ is $Gr$%
-$W$-$J_{gr}$-semiprime submodule of $\Im$.
\end{corollary}
\begin{proof}
Follows by Theorem \ref{th2.13} and Theorem \ref{th2.12}.
\end{proof}
%-------------------------------theorem 2.14---------------

\begin{theorem} \label{th2.15}
Let $U$ be $Gr$-small
submodule of $\Im $ with $J_{gr}(\Im )$ be a $Gr$-$W$-semiprime
submodule of $\Im $. Then $U$ is a $Gr$-$W$-$J_{gr}$-semiprime submodule of $%
\Im $.
\end{theorem}
\begin{proof}
Let $r_{g}\in h(\Re),$ $m_{h}\in
h(\Im)$ and $n\in
%TCIMACRO{\U{2124} }%
%BeginExpansion
\mathbb{Z}
%EndExpansion
^{+}$ with $0\neq r_{g}^{n}m_{h}\in U$. Since $U$ is a $Gr$-small submodule of $\Im$, then by \cite[Theorem 2.10]{KAl}, $U\subseteq J_{gr}(\Im)=\sum \{A$ $:$ $A$ is a $Gr$-small submodule
of $\Im\}$, so $0\neq r_{g}^{n}m_{h}\in J_{gr}(\Im)$, since $%
J_{gr}(\Im)$ is a $Gr$-$W$-semiprime submodule of $\Im$ then $r_{g}m_{h}\in
J_{gr}(\Im)\subseteq U+J_{gr}(\Im)$. Therefore $U$ is a $Gr$-$W$-$J_{gr}$%
-semiprime submodule of $\Im$.
\end{proof}
%-------------------------------------cor 2-15--------------------------
Recall from \cite{NV3} that a graded $\Re $-module $\Im $ is said to be a graded
finitely generated ($Gr$-finitely generated) if $\Im =\Re a_{g1}+$\textperiodcentered \textperiodcentered \textperiodcentered $+\Re a_{gn}$
for some $a_{g1},a_{g2},...,a_{gn}\in h(\Im ).$

%The following corollary follow directly from Theorem \ref{th2.15}
%
%\begin{corollary}\label{co2.16}
%Let $\Im $ be a $Gr$%
%-finitely generated $Gr$-multiplication $\Re$-module and $L$ be a $Gr$-$W$-$%
%J_{gr}$-semiprime ideal of $\Re$ with $ann_{\Re}(\Im )\subseteq L$, then $%
%L\Im $ is a $Gr$-$W$-$J_{gr}$-semiprime submodule of $\Im $.
%\end{corollary}
%-----------------------------theorem 2-17
\begin{theorem}\label{th2.17}
Let $\Im $ be a $Gr$-finitely generated $Gr$-multiplication $\Re$-module and $L$ be a $Gr$-$W$-$%
J_{gr}$-semiprime ideal of $\Re$ with $ann_{\Re}(\Im )\subseteq L$, then $%
L\Im $ is a $Gr$-$W$-$J_{gr}$-semiprime submodule of $\Im $.
\end{theorem}
\begin{proof}
Let $r_{g}\in h(\Re)$, $%
m_{h}\in h(\Im )$ and $k\in
%TCIMACRO{\U{2124} }%
%BeginExpansion
\mathbb{Z}
%EndExpansion
^{+}$ with $0\neq r_{g}^{k}m_{h}\in L\Im $. Then $\{0\}\neq r_{g}^{k}\langle m_{h}\rangle \subseteq L\Im .$ Since
$\Im $ is a $Gr$-multiplication so $\langle m_{h}\rangle =J\Im $ for some
$J\leq^{id} _{\Gamma}\Re $, hence $\{0\}\neq r_{g}^{k}J\Im \subseteq L\Im .$
This implies that $\{0\}\neq r_{g}^{k}J\subseteq L+ann_{\Re}(\Im )$ by \cite[Lemma 3.9]{AA}. Since $ann_{\Re}(\Im )\subseteq L$ it follows that $\{0\}\neq
r_{g}^{k}J\subseteq L$. Hence $r_{g}J\subseteq L+J_{gr}(\Re)$ as $L$ is a $%
Gr $-$W$-$J_{gr}$-semiprime ideal of $\Re.$ Thus $r_{g}J\Im \subseteq L\Im
+J_{gr}(\Re)\Im \subseteq L\Im +J_{gr}(\Im )$, and so $r_{g}m_{h}\in
r_{g}\langle m_{h}\rangle \subseteq L\Im +J_{gr}(\Im )$. Therefore, $L\Im $
is a $Gr$-$W$-$J_{gr}$-semiprime submodule of $\Im $.
\end{proof}
%--------------------------------------example 2-16---------------------------
The following example shows that the residual of $Gr$-$W$-$J_{gr}$-semiprime
submodule is not necessarily a $Gr$-$W$-$J_{gr}$-semiprime ideal.
\begin{example}
Let $\Gamma=%
%TCIMACRO{\U{2124} }%
%BeginExpansion
\mathbb{Z}
%EndExpansion
_{2}$ and $\Re=%
%TCIMACRO{\U{2124} }%
%BeginExpansion
\mathbb{Z}
%EndExpansion
$ be a $\Gamma$-graded ring such that $\Re_{0}=%
%TCIMACRO{\U{2124} }%
%BeginExpansion
\mathbb{Z}
%EndExpansion
$ and $\Re_{1}=\{0\}$. Let $\Im=%
%TCIMACRO{\U{2124} }%
%BeginExpansion
\mathbb{Z}
%EndExpansion
_{8}$ be a graded $\Re$-module such that $\Im_{0}=%
%TCIMACRO{\U{2124} }%
%BeginExpansion
\mathbb{Z}
%EndExpansion
_{8}$ and $\Im_{1}=\{\overline{0}\}$. Let $U=\{\overline{0},\overline{4}%
\}=\langle \overline{4}\rangle $ $\leq^{sub} _{\Gamma}\Im $. So it is
a $Gr$-$W$- $J_{gr}$-semiprime submodule of $\Im$ where $J_{gr}(\Im)=\langle
\overline{2}\rangle $. How ever $(U:_{\Re}\Im)=4%
%TCIMACRO{\U{2124} }%
%BeginExpansion
\mathbb{Z}
%EndExpansion
$ is not $Gr$-$W$-$J_{gr}$-semiprime ideal of $\Re,$ since $0\neq 2^{2}\cdot
1\in (U:_{\Re}\Im)$ where $2$, $1\in h(\Re)$ but $2\cdot 1\notin
(U:_{\Re}\Im)+J_{gr}(\Re)=4%
%TCIMACRO{\U{2124} }%
%BeginExpansion
\mathbb{Z}
%EndExpansion
+(0)=4%
%TCIMACRO{\U{2124} }%
%BeginExpansion
\mathbb{Z}
%EndExpansion
.$
\end{example}

%-----------------------theorem 2-17--------------------------
Next theorems show that the residual of $Gr$-$W$-$J_{gr}$-semiprime
submodule is a $Gr$-$W$-$J_{gr}$-semiprime ideal with under conditions.
\begin{theorem} \label{th2.19}
Let $\Im$ be a $Gr$-faithful $\Re$-module and $U\leq^{sub} _{\Gamma}\Im$ with $J_{gr}(\Im/U)=\{U\}$ and $
J_{gr}(\Re)\subseteq (U:_{\Re}\Im)$. Then $U$ is a $Gr$-$W$-$J_{gr}$-semiprime submodule of $\Im$ if and only if $
(U:_{\Re}\Im)$ is a $Gr$-$W$-$J_{gr}$-semiprime ideal of $\Re$.
\end{theorem}
\begin{proof}
.$(\Rightarrow )$Let $
a_{g}$, $b_{h}\in h(\Re)$ and $k\in
%TCIMACRO{\U{2124} }%
%BeginExpansion
\mathbb{Z}
%EndExpansion
^{+}$ with $0\neq a_{g}^{k}b_{h}\in (U:_{\Re}\Im)$. Hence $\{0\}\neq a_{g}^{k}b_{h}\Im\subseteq U$. Since $U$ is a $Gr$-$W$%
-$J_{gr}$-semiprime submodule of $\Im$, by Corollary \ref{co2.6}, we have $%
a_{g}b_{h}\Im\subseteq U+J_{gr}(\Im).$ By Theorem \ref{th2.8}, we have $%
J_{gr}(\Im)\subseteq U$ since $J_{gr}(\Im/U)=\{U\}.$ This implies that $%
a_{g}b_{h}\Im\subseteq U.$ Thus $a_{g}b_{h}\in (U:_{\Re}\Im)\subseteq
(U:_{\Re}\Im)+J_{gr}(\Re)$. Therefore, $(U:_{\Re}\Im)$ is a $Gr$-$W$-$J_{gr}$%
-semiprime ideal of $\Re$.

$(\Leftarrow )$ Let $r_{g}\in h(\Re)$ and $n\in
%TCIMACRO{\U{2124} }%
%BeginExpansion
\mathbb{Z}
%EndExpansion
^{+}$ with $\{0\}\neq \langle r_{g}\rangle ^{n}\Im\subseteq U$. Hence $\{0\}\neq \langle r_{g}\rangle ^{n}\subseteq (U:_{\Re}\Im)$ (if
$\langle r_{g}\rangle ^{n}=\{0\}$ then $\langle r_{g}\rangle ^{n}\Im=\{0\}$
as a contradiction), it follows that $0\neq r_{g}^{n}.1\in (U:_{\Re}\Im).$
Hence $r_{g}\cdot 1\in (U:_{\Re}\Im)+J_{gr}(\Re)$ as $(U:_{\Re}\Im)$ is a $%
Gr $-$W$-$J_{gr}$-semiprime ideal of $\Re.$ Since $J_{gr}(\Re)\subseteq
(U:_{\Re}\Im)$, we have $r_{g}\in (U:_{\Re}\Im)$, it follows that $\langle
r_{g}\rangle \subseteq (U:_{\Re}\Im)$. This yields that $\langle
r_{g}\rangle \Im\subseteq U\subseteq U+J_{gr}(\Im)$. Thus $U$ is a $Gr$-$W$ $%
J_{gr}$-semiprime submodule of $\Im$ by Corollary \ref{co2.5}.
\end{proof}
%-----------------------------------------------them 2-20
Recall that a graded $\Re$-module $\Im $ is said to be a graded cancellation ($Gr$-cancellation) if for any graded ideals $K$ and $L$ of $\Re$, $K\Im
=L\Im $, we have  $K=L$, see \cite{F}.
\begin{theorem}\label{th2.20}
Let $\Im$ be a $Gr$-finitely generated faithful $Gr$-multiplication $\Re$-module and $U$ $<^{sub}_{\Gamma}\Im$
with $J_{gr}(\Im)=J_{gr}(\Re)\Im$. Then $U$ is a $Gr$-$W$-$J_{gr}$-semiprime
submodule of $\Im$ if and only if $(U:_{\Re}\Im)$ is a $Gr$-$W$-$J_{gr}$-semiprime ideal of $\Re$.
\end{theorem}
\begin{proof}
$(\Rightarrow )$ Let $%
a_{g} $, $b_{h}\in h(\Re)$ and $k\in
%TCIMACRO{\U{2124} }%
%BeginExpansion
\mathbb{Z}
%EndExpansion
^{+}$ with $0\neq a_{g}^{k}b_{h}\in (U:_{\Re}\Im)$. Hence $\{0\}\neq a_{g}^{k}b_{h}\Im\subseteq U$.(if $%
a_{g}^{k}b_{h}\Im=\{0\} $ then $a_{g}^{k}b_{h}=0$ since $\Im$ \ is a
faithful as a contradiction). Since $U$ is a $Gr$-$W$-$J_{gr}$-semiprime
submodule of $\Im$, Corollary \ref{co2.6}, we get $a_{g}b_{h}\Im\subseteq
U+J_{gr}(\Im).$ This implies that $a_{g}b_{h}\Im\subseteq
(U:_{\Re}\Im)\Im+J_{gr}(\Im)$ as $\Im$ is $Gr$-multiplication module. Since $%
J_{gr}(\Im)=J_{gr}(\Re)\Im$, we have $a_{g}b_{h}\Im\subseteq
(U:_{\Re}\Im)\Im+J_{gr}(\Re)\Im=((U:_{\Re}\Im)+J_{gr}(\Re))\Im$, so $%
\langle a_{g}b_{h}\rangle \Im\subseteq ((U:_{\Re}\Im)+J_{gr}(\Re))\Im.$
Since $\Im$ is a $Gr$-finitely generated faithful $Gr$-multiplication by \cite[Theorem 2.10]{F}, we get $\langle a_{g}b_{h}\rangle \subseteq
(U:_{\Re}\Im)+J_{gr}(\Re)$. Hence $\langle a_{g}b_{h}\rangle
\subseteq (U:_{\Re}\Im)+J_{gr}(\Re)$, so $a_{g}b_{h}\in (U:_{\Re}\Im)+J_{gr}(\Re)$. Therefore, $(U:_{\Re}\Im)$
is a $Gr$-$W$-$J_{gr}$-semiprime ideal of $\Re$.

$(\Leftarrow )$ Let $\{0\}\neq r_{g}^{n}K\subseteq U$ where $r_{g}\in
h(\Re)$ and $K\leq^{sub} _{\Gamma}\Im $.  Since $\Im $ is a $Gr$-multiplication $\Re$-module, then there
exists nonzero $L\leq^{id} _{\Gamma}\Re $ with $K=L\Im $ it follows that
$\{0\}\neq r_{g}^{n}L\Im \subseteq U$ and hence $\{0\}\neq
r_{g}^{n}L\subseteq (U:_{\Re}\Im )$. So $r_{g}L\subseteq (U:_{\Re}\Im
)+J_{gr}(\Re)$ as $(U:_{\Re}\Im )$ is a $Gr$-$W$-$J_{gr}$-semiprime ideal of
$\Re.$ Hence $r_{g}L\Im \subseteq (U:_{\Re}\Im )\Im +J_{gr}(\Re)\Im
\subseteq (U:_{\Re}\Im )\Im +J_{gr}(\Im ).$ This implies that, $%
r_{g}K\subseteq U+J_{gr}(\Im)\ $as $\Im $ is a $Gr$-multiplication $\Re$%
-module. Thus, $U$ is a $Gr$-$W$-$J_{gr}$-semiprime submodule of $\Im$ by Corollary \ref{co2.6}.
\end{proof}
%--------------------------------them 2.2

\begin{theorem}
 Let $\Im$ a be $Gr$-finitely generated
faithful $Gr$-multiplication $\Re$-module and and $U$ $<^{sub}_{\Gamma}\Im$ with $%
J_{gr}(\Im)=J_{gr}(\Re)\Im$. Then the following statements are equivalent:
  \begin{enumerate}[\upshape (i)]
\item $U$ is a $Gr$-$W$-$J_{gr}$-semiprime submodule of $\Im$.

\item $(U:_{\Re}\Im)$ is a $Gr$-$W$-$J_{gr}$-semiprime ideal of $\Re$.

\item $U=L\Im $ for some a $Gr$-$W$-$J_{gr}$-semiprime ideal $L$ of $\Re$.
  \end{enumerate}
\end{theorem}
\begin{proof}
[(i)$\Rightarrow $(ii)]  By Theorem \ref{th2.20}

\bigskip \lbrack (ii)$\Rightarrow $(iii)] Since $\Im$ is a $Gr$%
-multiplication $\Re$-module, $U=(U:_{\Re}\Im)\Im$ \ where $%
(U:_{\Re}\Im)$ is a $Gr$-$W$- $J_{gr} $-semiprime ideal of $\Re$.

[(iii)$\Rightarrow $(i)] Let $U=L\Im $ for some $Gr$-$W$-$J_{gr}$-semiprime
ideal $L$ of $\Re$. Let $\{0\}\neq \langle r_{g}\rangle ^{n}\Im \subseteq U$
where $r_{g}\in h(\Re)$ and $n\in
%TCIMACRO{\U{2124} }%
%BeginExpansion
\mathbb{Z}
%EndExpansion
^{+}$, then $\{0\}\neq \langle r_{g}\rangle ^{n}\Im \subseteq L\Im .$ \
Since $\Im $ is a $Gr$-finitely generated faithful $Gr$-multiplication,
by \cite[Theorem 2.10]{F},  $\Im $ is
a $Gr$-cancellation. Thus $\{0\}\neq \langle r_{g}\rangle ^{n}\subseteq L$.
\ Since $L$ is a $Gr$-$W$-$J_{gr}$-semiprime ideal of $\Re$, $\langle
r_{g}\rangle \subseteq L+J_{gr}(\Re)$ it follows that $\langle r_{g}\rangle
\Im \subseteq L\Im +J_{gr}(\Re)\Im $. This yields that $\langle
r_{g}\rangle \Im \subseteq U+J_{gr}(\Im )$ since $J_{gr}(\Im
)=J_{gr}(\Re)\Im$. Therefore, $U$ is a $Gr$-$W$-$J_{gr}$-semiprime
submodule of $\Im $.
\end{proof}

%---------------------acknowledgment-------------------------------------------------
%----------------------------------------------------------------------------------
%
%
%  \noindent\textbf{Declarations}
%
%% \bigskip\bigskip
%%\noindent\textbf{Author contribution statement:}\\
%%{M. Alnimer: Conceived and designed the analysis; performed the experiments; Analyzed and interpreted the data; Contributed analysis tools or data; Wrote the paper.\\
%%{K. Al-Zoubi:  Conceived and designed the analysis; Analyzed and interpreted the data; Wrote the paper.\\
%%{M. Al-dolat : Analyzed and interpreted the data; Contributed analysis tools or data; Wrote the paper.
%
% \smallskip
%
% \noindent\textbf{Funding statement:}\\ This research did not receive any specific grant from funding agencies in the public, commercial, or not-for-profit sectors.
%%
% \smallskip
%
%
% \noindent\textbf{Data availability statement:} \\No data was used for the research described in the article
% \smallskip
%
% \noindent\textbf{Declaration of interests statement: }\\The authors declare no conflict of interest
% \smallskip
%
% \noindent\textbf{Additional information: } \\No additional information is available for this paper.
% \smallskip
%
% \noindent\textbf{Acknowledgments:} \\The authors wish to thank sincerely the referees for
% their valuable comments and suggestions.
%%\smallskip
%%

%-----------------------------------------------------------------------------------------

\bigskip\bigskip\bigskip\bigskip

\end{document}